\documentclass[12pt]{amsart}
\usepackage{amsmath,amsthm,amsfonts,amssymb,bm,graphicx, mathrsfs, fullpage}
\usepackage
{hyperref}

\usepackage[english]{babel}
\usepackage[margin=1in]{geometry}
\usepackage{todonotes}

\hypersetup{colorlinks=true,citecolor=blue,linkcolor=blue,urlcolor=blue,
pdfstartview=FitH }
\usepackage[all,cmtip]{xy}

\usepackage{microtype}
\usepackage{enumitem}

\usepackage[
backend=biber,
style=alphabetic,
]{biblatex}
\DeclareMathOperator{\per}{Per}

\theoremstyle{plain}
\newtheorem{theorem}{Theorem}

\newtheorem{lemma}[theorem]{Lemma}
\newtheorem{prop}[theorem]{Proposition} 
\newtheorem{cor}[theorem]{Corollary}

\theoremstyle{definition}

\newtheorem*{acknowledgments}{Acknowledgments}

\theoremstyle{remark}
\newtheorem{rem}{Remark}

\numberwithin{theorem}{section}
\numberwithin{equation}{section}
\numberwithin{definition}{section}

\begin{document}

\author{Asaf Katz}
\address{School of Mathematics, Georgia Institute of Technology, Atlanta GA 30332, USA}
\email{akatz47@gatech.edu}
\author{Thomas Aloysius O'Hare}
\address{Department of Mathematics, Northwestern University, Evanston IL 60203, USA}
\email{thomas.aloysius.ohare@northwestern.edu}

\title{Effective equidistribution for Contact Anosov flows in Dimension Three}

\begin{abstract}
    We prove effective equidistribution theorems for (weighted) packets of closed periodic orbits for Anosov flows. In particular, for the case of contact Anosov flows on three-dimensional manifolds, we show that the Bowen packets equidistribute at an exponential rate.
\end{abstract}

\maketitle
\section{Introduction}
Recall that an Anosov flow $f_{t}$ over a closed Riemannian manifold $M$ is defined as a flow $f_{t}:M\to M$ such that there exists $\lambda>1$ and a splitting of the tangent bundle $TM$ into a \emph{stable}, \emph{unstable} and central (flow) direction - $TM=E^{s}\oplus E^{0}\oplus E^{u}$ such that
$$ \left\lVert Df_1\mid_{ E^{s}}\right\rVert \leq \lambda^{-1}, \left\lVert Df_{1}\mid_{E^{u}}\right\rVert \geq \lambda, \dim E^{0}=1.$$
 
Let $\Psi:M\to\mathbb{R}$ be a $C^{1}$ function (known as the weight).
The topological pressure $P(\Psi)$ of the weight $\Psi$ is defined to be 
$$ P(\Psi) = \sup_{\nu} \left\{ \int \Psi d\nu + h_{\nu}(f_{1})\right\},$$
where the supremum is taken over the set of all $f_{t}$-invariant probability measures $\nu$ and $h_{\nu}(f_{1})$ is the measure theoretical entropy of the time-one map $f_{1}$ with respect to $\nu$.
We assume that $P(\Psi)>0$.
A measure $\mu$ is called an equilibrium state (associated to a weight $\Psi$) if 
$$ P(\Psi) = \int \Psi d\mu + h_{\mu}(f_{1}),$$
namely the measure is an extermizer of the pressure function.

It is well known that in the case of Anosov flows, such equilibrium states are unique (c.f.~\cite[Theorem~7.3.6]{Fisher_Hasselblatt_2019}).

Let $\per((0,T])$ denote the set of closed orbits under $f_t$ of period at most $T$ and for a periodic orbit $\tau$ denote its period by $\ell(\tau)$. Consider the family of measures supported on closed periodic orbits for the $f_{t}$-flow 
\begin{equation}
\label{eqn: Discrete Measure def}
    \mu_{\Psi,T}:=\frac{1}{Z_T(f_t,\Psi)}\sum_{\tau\in\per((0,T])}e^{\int_\tau \Psi }\delta_\tau,
\end{equation}
where $Z_T(f_t,\Psi)$ is a normalization constant to make $\mu_{\Psi,T}$ a probability measure. It is a theorem of R. Bowen~\cite{Bowen_1972} that these discrete measures converge to the measure of maximal entropy in the weak$^*$-topology: $\mu_{\Psi,T}\rightarrow_{w^*}\mu_\Psi$.
Our goal is to prove a quantitative version of this equidistribution result.

A \emph{contact} Anosov flow is an Anosov flow which preserves a contact form, i.e. a $1-$form $\alpha$ such that $\alpha\wedge(d\alpha)^n\neq0$, where $\dim(M)=2n+1$.
A prime example of a contact Anosov flow is the geodesic flow associated to a closed surface $N$ of strictly negative curvature $g_{t}:T^{1}N\to T^{1}N$. In fact, 
In this case, the measure of maximal entropy is known as the BMS - Bowen-Margulis-Sullivan measure\footnote{This measure coincides with the Liouville measure in the case of constant negative curvature.}. In this important case, we can provide \emph{exponential rate of convergence}:

\begin{theorem}[Effective Equidistribution for contact flows]
\label{thm: Effective Equidistribution}
    Let $g_t:M\rightarrow M$ be a contact Anosov flow on a three-dimensional manifold $M$, and let $\Psi\in C^1(M)$ be a potential with equilibrium state $\mu_\Psi$ and $P(\Psi)>0$, and let $\mu_{\Psi,T}$ be as in (\ref{eqn: Discrete Measure def}). Then there exist constants $C,\delta>0$ such that for any $C^1$ function $K:M\rightarrow\mathbb{R}$ we have
    \begin{equation}
    \label{eqn: Effective Equidistribution}
        \left|\int K d\mu_\Psi-\int K d\mu_{\Psi,T} \right|\leq
        C||K||_{C^1}e^{-\delta\cdot P(\Psi)\cdot T}.
    \end{equation}
\end{theorem}
This result covers the fundamental case of geodesic flow over a compact surface of negative curvature but also proves equidistribution for other types of three dimensional contact flows~\cite{Foulon_Hasselbaltt_2013}.
\begin{rem}
    The hypothesis that $\Psi$ is $C^1$ can be weakened to assuming only H\"older regularity, but the method of proof we present relies on smoothness of the potential. To circumvent this, one could approximate a H\"older $\Psi$ by $C^1$ functions, following closely the approach of Dolgopyat in \cite{Dolgopyat_1998}. However, as the main potentials we are interested in $\Psi\equiv 0$ and the geometric potential $\Psi_u=-\frac{\partial}{\partial t}\log(Df_t|_{E^u})|_{t=0}$, corresponding respectively the measure of maximal entropy $\mu_{MME}$ and the SRB measure $\mu_{SRB}$, are both $C^1$, we will not concern ourselves with the details for general H\"older potentials. Notice that while our theorem does not apply directly to the geometric potential, since $P(\Psi_u)=0$, we can change our potential by adding to it any positive constant $\varepsilon>0$ so that $P(\Psi_u+\varepsilon)=\varepsilon>0$, and note that that equilibrium state of $\Psi_u+\varepsilon$ is still the SRB measure (though the discrete measures (\ref{eqn: Discrete Measure def}) do change when we add a constant to our potential).
\end{rem}

\begin{rem}
    It is possible that the result and its proof are valid also in the case of non-compact surfaces. Nevertheless, as the non-compact setting introduces technical complications, we decided to state and prove the theorem only for the closed case.
\end{rem}

In the case of general Anosov flow over a closed manifold $f_{t}:M\to M$, our results are weaker and we can only achieve \emph{polynomial rate}:
\begin{theorem}[Effective equidistribution for Anosov flows]
\label{thm: Effective Equidistribution - general Anosov}
    Let $f_{t}:M\to M$ be a transitive weak-mixing Anosov flow, and let $\Psi\in C^1(M)$ be a potential with equilibrium state $\mu_\Psi$ and $P(\Psi)>0$, and let $\mu_{\Psi,T}$ be as in (\ref{eqn: Discrete Measure def}). Then there exist constants $C,\delta>0$ such that for any Lipschitz function $K:M\rightarrow\mathbb{R}$ we have
    \begin{equation}
    \label{eqn: Effective Equidistribution poly}
        \left|\int K d\mu_\Psi-\int K d\mu_{\Psi,T} \right|\leq
        C||K||_{Lip}T^{-\delta}
    \end{equation}
\end{theorem}

Broadly speaking, the proof of Theorem \ref{thm: Effective Equidistribution} takes place in three main steps. The first is appropriately code our system in a particular way that allows us to utilize the smoothness of our dynamics. In effect, for the case of contact Anosov flows - we will get a ``$C^1$-coding." This coding is detailed in Dolgopyat's thesis, and the paper of Pollicott and Sharp, in the case of geodesic flow on the unit tangent bundle of a negatively curved surface. 

The second step is to prove effective equidistribution of the coded system. Following the approach of Pollicott and Sharp in~\cite{Pollicott-Sharp_1998}, we prove effective equidistribution by relating it to the analytic extensions of certain dynamical zeta functions using techniques from analytic number theory. The main difference is that we use the two-variable zeta function discussed by Parry and Pollicott~\cite{Parry_Pollicott_1990} which incorporates the test function $K$.

The final step is to pass from effective equidistribution on the symbolic level to effective equidistribution on the dynamical system itself. This involves handling deficiencies in the counting introduced by our coding.

\subsection{Related results}
Bowen's proof of his theorem, relied on a uniqueness principle for measures of maximal entropy, developing earlier approach of Adler and Weiss. This approach is essentially quantifiable (by controlling the relative entropy, the KL-divergence and then bounding the difference of the integrals via Pinsker inequality), but yields sub-optimal approaches. This was carried, in very similar setups by I. Khayutin~\cite{Khayutin_2017} and R. Ruhr~\cite{Ruhr_2016}. Additional related work was carried out for the specific case of subshifts by S. Kadyrov~\cite{Kadyrov_2017} for the measure of maximal entropy, and the second-named author~\cite{O’Hare_2025} for general equilibrium states.

In the case of hyperbolic surfaces of \emph{constant negative curvature}, one may apply spectral theory to analyze this problem.
In pioneering works, S. Zelditch~\cite{Zelditch_1989,Zelditch_1992} managed to obtain a complete \emph{spectral resolution} for the measures $\mu_{T}$, by using a \emph{non-spherical} Trace formula. 
Using his spectral resolution, Zelditch managed to prove \emph{exponential decay rate}, with essentially the best possible rates, for \emph{specific functions}\footnote{In principle, Maass forms and lowest weight vectors for discrete series representations.}. Zelditch did not manage to extend his results for general smooth functions, as some triple product bounds were not known back then. Even with these bounds now at hand, bookkeeping the constants involved in the Mellin inversion approach of Zelditch is not easy and requires handling some special functions. Zelditch managed to only conclude \emph{super polynomial} bound for general smooth functions.
We hope to revisit Zelditch's approach in the future.

In a different work, Margulis-Mohammadi-Oh~\cite{MargulisMohammadiOh_2014} managed to obtain such equidistribution results in the context of counting results with holonomy constraints. Their approach uses Margulis' flow-box argument.

Sarnak-Wakayama~\cite{SarnakWakayama_1999} studied the counting problem with holonomy constraints earlier using a non-spherical trace formula based approach.

Our approach uses the study of (modified) Ruelle Zeta functions and transfer operators for such flows.
The initial coding results are classical, mostly attributed to Bowen and Ratner.
Related approaches to this equidistribution problem were developed by Parry-Pollicott.
An important breakthrough for Anosov flows and geodesic flows was achieved by D. Dolgopyat proving mixing in large generality. 

Lastly, Giulietti-Liverani-Pollicott~\cite{Giulietti_2013} managed to improve some of the available results in the spirit of Dolgopyat to general Anosov flows. It seems that the results that we need (zero-free strips) can be inferred from their paper for small perturbations of geodesic flows over constant curvature surfaces. We hope to consider this case in the future.

\subsection{Structure of the paper}
In \S\ref{sec:coding}, we discuss a coding procedure for the flow, mimicking Dolgopyat's own work.
In \S\ref{sec:effective-geodesic} we present an analytic framework, based on two-variable Zeta functions, to study equidistribution in the encoded system, for contact Anosov flows. This involves using Dolgopyat's inequality to get an effective Phragm\'en-Lindel\"of argument. 
Moreover, we later show that an effective equidistribution result for the encoded system translates to an effective result for the original system.
In \S\ref{sec:effective-general-Anosov}, we show how to modify the technique of the previous section to handle general Anosov flows, at the expanse of getting only a polynomial rate.

\begin{acknowledgments}
It is a pleasure to thank Dimitry Dolgopyat, Andrei Gogolev, Ralf Spatzier, Amir Mohammadi, Stephen Cantrell and Caleb Dilsavor for very helpful discussion regarding equidistribution results of periodic orbits.
\end{acknowledgments}

\section{Encoding Anosov systems}\label{sec:coding}
A standard reference for the material below is~\cite{Fisher_Hasselblatt_2019}, especially $\S 6$.

\subsection{The Bowen-Ratner coding}
The basic result we start with is the standard coding by a suspension flow over a subshift of finite type using Markov partitions:
\begin{prop}[Bowen-Ratner coding, \cite{Bowen_1973, Ratner_1973},\cite{Fisher_Hasselblatt_2019} Theorem 6.6.5]
\label{prop: Bowen Ratner Coding}
    There exists and irreducible and aperiodic transition matrix $A$ and a strictly positive H\"older continuous function $r:\Sigma_A\rightarrow\mathbb{R}$ and a surjective continuous map $\pi:\Sigma_A^r\rightarrow M$ such that
    \begin{enumerate}
        \item $\pi\circ \sigma_t^r=f_t\circ\pi$
        \item Every closed $f_t$-orbit corresponds to $\sigma_t^r$-orbit of the same (prime) period, with at most finitely many exceptions.
    \end{enumerate}
\end{prop}

The usual proof is to build transversals to the flow satisfying a Markov property, however for our application we will need to do an additional refinement originally due to Dolgopyat~\cite{Dolgopyat_1998}. We choose a finite number of co-dimension one smooth transversals $\mathcal{T}_1,\cdots,\mathcal{T}_N$ to the flow $f_t$ and set $\mathcal{T}=\bigcup_{i=1}^N\mathcal{T}_i$. Then we have an associated Poincar\'e map $P:\mathcal{T}\rightarrow\mathcal{T}$ with return time map $R:\mathcal{T}\rightarrow\mathbb{R}$ such that $P(x)=f_{R(x)}(x)$ for all $x\in\mathcal{T}$. Our symbolic coding space $\Sigma_A^R$ then has the transition matrix $A$ determined by $P:\mathcal{T}\rightarrow\mathcal{T}$ and roof function $R$. 

We note that the above construction also shows that the stable and unstable boundaries of the Markov partition used in the coding are invariant under first return time map of the flow~\cite[Lemma~6.6.7]{Fisher_Hasselblatt_2019}.

By adjusting each $\mathcal{T}_i$ in the flow direction, we produce a new collection of codimension one cross-sections $\Tilde{\mathcal{T}}=\bigcup_{i=1}^N\Tilde{\mathcal{T}}_i$ such that:
\begin{enumerate}[label=(\roman*)]
    \item each $\Tilde{\mathcal{T}}_i$ is foliated by local stable manifolds $W^s_\varepsilon(x)$, and
    \item each $\Tilde{\mathcal{T}}_i$ contains a piece of local unstable manifold, which we denote by $I_i$.
\end{enumerate}
Note that these adjustments do not change the underlying shift space $(\Sigma_A,\sigma)$. We denote the new roof function by $r:\Tilde{\mathcal{T}}\rightarrow\mathbb{R}$. Observe that if $y\in W^s_\varepsilon(x)\cap\Tilde{T}_i$ then $r(y)=r(x)$. Indeed, since the stable and flow directions are jointly integrable we have 
$$f_{r(x)}(W^{s}_\varepsilon(x)\cap~\Tilde{T}_{i})\subset~W^{s}_{\varepsilon}(f_{r(x)}(x))=~W^{s}_\varepsilon(P(x)).$$ Thus, if we collapse the stable foliations to the unstable interval $I_i$, then the roof function $r:\Tilde{\mathcal{T}}\rightarrow\mathbb{R}$ gives rise to a roof function over the union of the $I:=\bigcup_{i=1}^N I_i$, which we also denote by $r$:
\begin{equation*}
    r:I\rightarrow\mathbb{R}.
\end{equation*}
Likewise, the Poincar\'e map gives rise to an expanding map 
\begin{equation*}
    f:I\rightarrow I.
\end{equation*}
We note here the critical fact - as the manifolds $W^{u},W^{s}$ are as smooth as the flow $f_{t}$~\cite[\S6, Theorem~$6.I$]{Ruelle_1979}, the resulting intervals $I_{i}$ are as smooth as the flow $f_{t}$.

\subsection{Coding flows with and without $C^{1}$-foliations}\label{subsec:lifting-functions}
We now focus on the case of a contact Anosov flow $g_{t}:M\to M$ on a three-dimensional manifold $M$.
It is well-known, going back to Hopf, that in this case of geodesic flow on a negatively curved surface, the foliations are $C^{1}$. Actually, the foliations are much smoother, by the smoothness of the Busemann cocycle~\cite[\S3]{Heintze_Hof_1977}. For a $C^{2}$ contact Anosov flow with contact form $\alpha$, it is well-known that $\ker(\alpha)=E^s\oplus E^u$, and hence the joint stable-unstable bundle is smooth. Furthermore, by Katok-Hurder \cite{Hurder_Katok_1990}, we have that $E^s\oplus E^0$ and $E^0\oplus E^u$ are both $C^{1+\varepsilon}$. By intersecting these bundles, we find that $E^s$ and $E^u$ are both $C^{1+\varepsilon}$. One may consult~\cite{Foulon_Hasselblatt_2003} for more refined results.

This in turn allows us to ``transfer'' $C^{1}$-functions between the original system and the symbolic space.
Whenever the foliations are $C^{1}$, we have the following modification to the encoding given by the Markov partition, essentially due to Dolgopyat, as explained by N. Anantharman~\cite[\S4]{Anantharaman_2000}.
As the intervals $I_i$ are as smooth as the flow $g_t$, whereas the maps $f$ and $r$ are only as smooth as the stable and unstable foliations, which in our setting are $C^1$. We can therefore code the contact Anosov flow $g_t$ by a $C^1$ expanding map $f:I\rightarrow I$ together with a $C^1$ roof function $r:I\rightarrow\mathbb{R}$. We choose $0<\gamma<1$ such that $|f'(x)|\geq\gamma^{-1}$.

One of the most important properties of our coding is that it allows us to lift $C^1$ functions on $M$ to $C^1$ functions on $I$:
\begin{lemma}[$C^{1}$ lifting Lemma]
\label{lem: Lifting C1 functions to coding}
    Let $K\in C^1(M)$. Then $k\in C^1(I)$, where for all $x\in I$:
    \begin{equation*}\label{eq:k-defn}
        k(x)=\int_0^{r(x)}K(f_s(x))ds.
    \end{equation*}
    Moreover, there exists a constant $C>0$ depending only on $g_t$ such that $\left\lVert k\right\rVert_{1,t}\leq C\left\lVert K\right\rVert_{C^1}$.
\end{lemma}
\begin{proof}
    The continuity and smoothness follow easily from the definition together with the smoothness of the roof function $r(x)$.
    For the norm bound, we have the following trivial bound:
    \begin{equation*}
       \begin{split}
            \sup_{x}\left\lvert k(x)\right\rvert &\leq \left\lVert r\right\rVert_{C^0}\cdot \left\lVert K\right\rVert_{C^0} \\
            &\leq \left\lVert r\right\rVert_{C^0}\cdot \left\lVert K\right\rVert_{C^1}.
       \end{split}
    \end{equation*}
    Moreover, differentiating the integral gives the following:
    \begin{equation*}
        k'(x) = K \left(f_{r(x)}x\right )\cdot r'(x) + \int_{0}^{r(x)}K'\left (f_{s}x\right)ds,
    \end{equation*}
    hence we get
    \begin{equation*}
    \begin{split}
        \left\lvert k'(x)\right\rvert &\leq \left\lVert r'\right\rVert_{C^0}\cdot \left\lVert K\right\rVert_{C^0} + \left\lVert r\right\rVert_{C^0}\cdot \left\lVert K'\right\rVert_{C^0} \\
        &\leq 2\left\lVert r\right\rVert_{C^1}\cdot \left\lVert K\right\rVert_{C^1}.
    \end{split}
    \end{equation*}
\end{proof}

For an Anosov flow $f_{t}:M\to M$, the roof function $r:I\to\mathbb{R}$ is generally only $\alpha_0$-H\"older continuous for some $0<\alpha_0\leq1$.
We remark that examining the proof of the previous Lemma gives the following H\"older lifting Lemma:

\begin{lemma}[H\"older lifting Lemma]
    Let $K\in Lip(M)$. Then $k\in C^{\alpha_0}(I)$ where $k$ is defined as before and $0<\alpha_0\leq1$ is the H\"older exponent of the roof function $r$. Moreover, we have the estimate - $\left\lVert k\right\rVert_{C^\alpha(I)}\ll_{r} \left\lVert K\right\rVert_{Lip(M)}.$
\end{lemma}
\begin{proof}
    The supermum bound from Lemma~\ref{lem: Lifting C1 functions to coding} holds verbatim.
    We replace the derivative bound as follows
    \begin{equation*}
    \begin{split}
        \left\lvert k(y)-k(x)\right\rvert &\leq \int_{r(x)}^{r(y)}\left\lvert K\left(f_{s}(x)\right)\right\rvert ds \\
        &\leq \left\lVert K \right\rVert_{\infty} \cdot \left\lvert r(y)-r(x)\right\rvert. 
    \end{split}
    \end{equation*}
    The second factor is clearly dominated by $\left\lVert r\right\rVert_{C^{\alpha_0}(I)}\cdot d(x,y)^{\alpha_0}$.
\end{proof}

\section{Effective Equidistribution of Contact Anosov Flows via Zeta Functions}\label{sec:effective-geodesic}
The goal of this section is to proof Theorem \ref{thm: Effective Equidistribution}.
Throughout this section we denote $s=\sigma+it$. 
We assume the results of the previous section, \S\ref{sec:coding}. In particular, we assume we have a finite disjoint collection of $N$ intervals $I=\bigcup_{i=1}^{N} I_{i}$ and differentiable function $f:I\to I$ which is an expanding map, satisfying $\left\lvert f'(x)\right\rvert \geq \gamma^{-1}$ for all $x\in I$.

We start by defining the following norm on $C^1(I)$:

\begin{equation}
\label{eqn: Norm on C1(I)}
     ||h||_{1,t}=\max\left\{||h||_{C^0},\frac{||h'||_{C^0}}{\left\lvert t\right\rvert} \right\},\hspace{1mm} \text{if}\hspace{1mm} \left\lvert t\right\rvert>1,
\end{equation}
and 
\begin{equation}
    ||h||_{1,t}=\max\left\{||h||_{C^0},||h'||_{C^0} \right\},\hspace{1mm} \text{if}\hspace{1mm} \left\lvert t\right\rvert\leq 1.
\end{equation}
For simplicity, we denote $c:=P(\Psi)>0$. Of particular importance to us is the transfer operator 
$\mathcal{L}_{\psi-scr}:C^1(I)\rightarrow C^1(I)$ defined by 
\begin{equation}
    \mathcal{L}_{\psi-scr}w(x)=\sum_{f(y)=x}e^{\psi(y)-scr(y)}w(y),
\end{equation}
where $w\in C^1(I)$ is a differentiable function defined over the coding space $I$, and $\psi\in C^1(I)$ is the lift of $\Psi$ given by lemma \ref{lem: Lifting C1 functions to coding}. Note that the reason we take $\Psi\in C^1(M)$ in theorem \ref{thm: Effective Equidistribution} is so that the the operator $\mathcal{L}_{\psi-scr}$ preserves the Banach space $C^1(I)$.

The key estimate in our analysis is the following bound on the spectral radius of $\mathcal{L}_{-sr}$ in terms of the parameter $t$ due to Dolgopyat~\cite[Corollary~$3$]{Dolgopyat_1998}.
A standard reference in the setting of hyperbolic flows is given in~\cite[Section~$\S7$]{Fisher_Hasselblatt_2019}.
\begin{prop}[Dolgopyat's estimate~\cite{Dolgopyat_1998}~Corollary $3$, see also \cite{Anantharaman_2000}~Theorem~$4.5$]
\label{prop: Dolgopyat Estimate}
    There exist constants $\sigma_0<1$, $C>0$, and $0<\rho<1$ such that 
    whenever $s=\sigma+it$ with $\sigma\geq\sigma_0$ and $\left\lvert t\right\rvert\geq 1$, and $n=p\lfloor\log\left\lvert t\right\rvert\rfloor+l$, where $p\geq0$ and $0\leq l\leq \log\left\lvert t\right\rvert-1$, then we have
    \begin{equation}
    \label{eqn: Dolgopyat Estimate}
        \left\lVert \mathcal{L}_{\psi-scr}^{n} \right\rVert _{1,t}\leq C\rho^{p\lfloor\log\left\lvert t\right\rvert\rfloor}e^{lP(\psi-\sigma cr)}.
    \end{equation}
\end{prop}
\begin{rem}
    Proposition \ref{prop: Dolgopyat Estimate} appears in \cite{Dolgopyat_1998} in the case where $g_t:T^{1}N\to T^{1}N$ is the geodesic flow on a negatively curved surface $N$, however the main ingredients we need are that the strong stable and unstable bundles are both $C^1$, and the equilibrium state $\mu_{\Psi}$ is a Federer measure, i.e., satisfies a doubling condition $\mu_{\Psi}(B(x,2r))\leq C\mu_{\Psi}(B(x,r))$ for some uniform $C>0$ and all $r>0$. As discussed in Section \S\ref{subsec:lifting-functions}, the bundles $E^s,E^u\in C^1$ when $g_t$ is contact Anosov, and the Federer property holds in general for all \emph{Gibbs measures} when the stable and unstable foliations are $1$-dimensional. In particular, both conditions are met when $\dim(M)=3$ and $g_t$ is contact.
\end{rem}
\noindent From now on, when $s\in\mathbb{C}$, we will write $s=\sigma+it$ to represent its decomposition into its real and imaginary parts.

Before discussing the two-variable zeta function, we first recall the standard zeta function associated with the smooth expanding map $f:I\rightarrow I$ and roof function $r:I\rightarrow\mathbb{R}$:
\begin{equation}
    \label{eqn: Standard Zeta Function}
    \zeta_{\Psi}(s):=\exp\left(\sum_{n=1}^\infty\frac{1}{n}\sum_{f^n(x)=x}e^{\psi^{(n)}(x)-scr^{(n)}(x)}\right).
\end{equation}
Here we use the notation $r^{(n)}(x)=r(x)+r(f(x))+\cdots+r(f^{n-1}(x))$, when $f^{(n)}(x)=x$. It is standard that $\zeta_{\Psi}(s)$ is analytic on the half-plane $\Re(s)\geq 1$, except for a simple pole at $s=1$. It follows from Proposition \ref{prop: Dolgopyat Estimate}, as in Pollicott and Sharp~\cite{Pollicott-Sharp_1998}, that there exists $\varepsilon>0$ such that $\zeta_{\Psi}(s)$ has an analytic and zero-free extension to the half-plane $\Re(s)\geq 1-\varepsilon$, except, of course, for the simple pole at $s=1$. Analytic extensions of zeta functions can be related to asymptotic counting formulas for periodic orbits. We will make heavy usage of the (non-trivial) fact that the extension is \emph{zero free} in the strip $1>\Re(s)\geq 1-\epsilon$.

Now, let $K\in C^1(M)$. 
We will construct a new zeta function, taking into account the function $K$, whose analytic extension is related to the effective equidistribution of the function $K$. 

To prove Theorem \ref{thm: Effective Equidistribution}, we will use the following zeta function for $K\in C^1(M)$:
\begin{equation}
    \label{eqn: Two Variable Zeta Function}
    \zeta_{\Psi,K}(s,z):=\exp\left(\sum_{n=1}^\infty\frac{1}{n}\sum_{f^n(x)=x}e^{\psi^{(n)}(x)-scr^{(n)}(x)+zk^{(n)}(x)}\right),
\end{equation}
where the function $k$ is related to $K$ by the formula in~\eqref{eq:k-defn}.
Actually, we are more interested in its logarithmic derivative in the second variable at $z=0$:
\begin{equation}
    \label{eqn: Log Derivative of Zeta Function}
    \begin{split}
    \eta_{\Psi,K}(s) &:=\frac{\partial}{\partial z}\log\zeta_{\Psi,K}(s,z)\bigg|_{z=0}\\
    &=\sum_{n=1}^\infty\frac{1}{n}\sum_{f^n(x)=x}k^{(n)}(x)e^{\psi^{(n)}(x)-scr^{(n)}(x)}.
    \end{split}
\end{equation}
The basic theorem about the analytic extension of $\eta_K$ comes from Parry and Pollicott~\cite{Parry_Pollicott_1990}:
\begin{prop}
\label{prop: Initial Domain of eta}
    $\eta_{\Psi,K}(s)$ is analytic on the half plane $\Re(s)\geq 1$, except for a simple pole at $s=1$. The residue of the pole at $s=1$ is equal to $\frac{1}{c}\int K d\mu$ whenever $k$ arises from $K\in C^1(M)$ as in Lemma~\ref{lem: Lifting C1 functions to coding}.
\end{prop}
\begin{rem}
    Technically speaking, Parry and Pollicott define their zeta function and prove this proposition for suspensions over subshifts of finite type $\Sigma_A^r$, with the inner sum defined over $\sigma^n(x)=x$ instead of $f^n(x)=x$. However, since there is a one-to-one correspondence between periodic orbits of $\sigma$ and $f$, since and $r$ is constant along local stable leaves, the two zeta functions are the same.
\end{rem}

The main technical step in proving Theorem~\ref{thm: Effective Equidistribution} is to improve Proposition~\ref{prop: Initial Domain of eta} by extending $\eta_K$ to a strictly larger half plane (uniformly larger in terms of $K$).
\begin{theorem}[Growth estimate for two-variable zeta functions]
\label{thm: Uniform extension of eta}
    There exists $\sigma<1$ such that for every $K\in C^1(M)$, $\eta_{\Psi,K}(s)$ has an analytic extension to the half plane $\Re(s)\geq \sigma$. Moreover, we have the bound
    \begin{equation}
        \label{eqn: Bounds of eta}
        |\eta_{\Psi,K}(\sigma+it)|=O\left(\left\lVert K\right\rVert_{C^1}\left\lvert t\right\rvert^\alpha\right),
    \end{equation}
    as $\left\lvert t\right\rvert\rightarrow\infty$ for some $0<\alpha<1$.
\end{theorem}

We will use the notation
    \begin{equation*}
        Z_n(\psi-scr,k):=\sum_{f^n(x)=x}k^{(n)}(x)e^{\psi^{(n)}(x)-scr^{(n)}(x)}.
    \end{equation*}
Observe that we have the relation
    \begin{equation*}
        \eta_{\Psi,K}(s)=\sum_{n=1}^\infty\frac{1}{n} Z_n(\psi-scr,k).
    \end{equation*}
This follows directly from the definition of $\eta_K$ in~\eqref{eqn: Log Derivative of Zeta Function}. Next, we relate $Z_n(\psi-scr,k)$ to the transfer operator $\mathcal{L}_{\psi-scr}$.
Recall that from Proposition~\ref{prop: Dolgopyat Estimate} we have a parameter $\rho>0$. Furthermore, under the coding assumption, we have $\gamma\in (0,1)$ such that $\left\lvert f'\right\rvert >\gamma^{-1}$.

The next lemma establishes a key inequality relating the zeta function to the repeated application of the transfer operator.

\begin{lemma}[Ruelle-type estimate]
\label{lem: Key Bound for Eta}
    For \emph{any} $\rho_0$ such that $\max\{\rho,\gamma\}<\rho_0<1$, there exist $C_{1},\epsilon$ and $0<\rho_1<1$ such that for points $x_i\in I_i$ and $\left\lvert t\right\rvert\geq 1$ we have

    \begin{equation}
    \label{eqn: Key Bound for Eta}
        \left|Z_n(\psi-scr,k)-\sum_{i=1}^N\mathcal{L}_{\psi-scr}^n(\chi_{I_{i}}k)(x_i) \right|\leq C_1 \lVert k\rVert _{1,t}\cdot \left\lvert t\right\rvert^{1+\varepsilon}\cdot n\cdot \rho_1^n.
    \end{equation}
\end{lemma}
\begin{proof}

We say the a string of length $n$, $\underline{i}=(i_0,\cdots,i_{n-1})$ ($1\leq i_j\leq N$), is \emph{admissible string} if $I_{i_0}\cap f^{-1}(I_{i_1})\cap\cdots\cap f^{-(n-1)}(I_{i_{n-1}}) \neq \emptyset$. We denote $I_{\underline{i}}:=I_{i_0}\cap f^{-1}(I_{i_1})\cap\cdots\cap f^{-(n-1)}(I_{i_{n-1}})$ and write $\lvert \underline{i}\rvert=n$. If $\underline{i}=(i_0,\cdots,i_{m-1})$, we use the notation $\underline{j}(\underline{i}):=~(i_0,\cdots,i_{m-2})$, i.e., the string obtained by removing the last entry of $\underline{i}$. 
For each admissible string $\underline{i}$ we choose points $x_{\underline{i}}\in I_{\underline{i}}$ in the following manner. If $\left\lvert\underline{i}\right\rvert = 1$, we fix $x_{\underline{i}}=x_i$, for the given points $x_i\in I_i$. 
For admissible strings of length greater than one, we choose $x_{\underline{i}}\in I_{\underline{i}}$ such that $f^{\left\lvert \underline{i}\right\rvert}(x_{\underline{i}})=x_{\underline{i}}$ if possible and if not we let $x_{\underline{i}}$ to be an arbitrary point in $I_{\underline{i}}$.

With this notation, we can easily see that

\begin{equation}
     Z_n(\psi-scr,k)=\sum_{|\underline{i}|=n}\mathcal{L}_{\psi-scr}^n(\chi_{I_{\underline{i}}}k)(x_{\underline{i}}).
\end{equation}

To obtain estimate (\ref{eqn: Key Bound for Eta}), we perform a standard telescoping trick:
\begin{equation}\label{eqn: Telescoping sum trick}
    \begin{split}
        \left|Z_n(\psi-scr,k)-\sum_{i=1}^N\mathcal{L}_{\psi-scr}^n(\chi_{I_{i}}k)(x_i) \right| & =\left|\sum_{|\underline{i}|=n}\mathcal{L}_{\psi-scr}^n(\chi_{I_{\underline{i}}}k)(x_{\underline{i}})-\sum_{i=1}^N\mathcal{L}_{\psi-scr}^n(\chi_{I_{i}}k)(x_i) \right|
    \\
    &= \left|\sum_{m=2}^{n}\sum_{|\underline{i}|=m}\mathcal{L}_{\psi-scr}^n(\chi_{I_{\underline{i}}}k)(x_{\underline{i}})-\mathcal{L}_{\psi-scr}^n(\chi_{I_{\underline{i}}}k)(x_{{\underline{j}(\underline{i})}})
    \right|\\
    &\leq 
    \sum_{m=2}^n \lVert\mathcal{L}_{\psi-scr}^{n-m}\rVert_{1,t}\cdot \lvert t\rvert \cdot \sum_{|\underline{i}|=m}\lVert\mathcal{L}_{\psi-scr}^m(\chi_{I_{\underline{i}}}k)\rVert_{1,t} \cdot d(x_{\underline{i}},x_{\underline{j}(\underline{i})}).
    \end{split}
\end{equation}

To see where the factor of $\left\lvert t\right\rvert$ comes from in the last expression, we observe that
\begin{equation*}
    \lvert\mathcal{L}_{\psi-scr}^n(\chi_{I_{\underline{i}}}k)(x_{\underline{i}})-\mathcal{L}_{\psi-scr}^n(\chi_{I_{\underline{i}}}k)(x_{{\underline{j}(\underline{i})}})\rvert \leq \lVert \mathcal{L}_{\psi-scr}^n(\chi_{I_{\underline{i}}}k)\rVert_{C^1} \cdot d(x_{\underline{i}},x_{\underline{j}(\underline{i})}),
\end{equation*}
and by the definition of the $\lVert\cdot\rVert_{1,t}$ norm

\begin{equation*}
\begin{split}
    \lvert \mathcal{L}_{\psi-scr}^n(\chi_{I_{\underline{i}}}k)\rvert_{C^1}&\leq \lVert\mathcal{L}_{\psi-scr}^n(\chi_{I_{\underline{i}}}k)\rVert_{1,t} \cdot \lvert t\rvert\\
    &= 
    \lVert\mathcal{L}_{\psi-scr}^{n-m}(\mathcal{L}_{\psi-scr}^m(\chi_{I_{\underline{i}}}k))\rVert_{1,t}\cdot \lvert t\rvert \\
    &\leq
    \lVert\mathcal{L}_{\psi-scr}^{n-m}\rVert_{1,t} \cdot \lVert\mathcal{L}_{\psi-scr}^m(\chi_{I_{\underline{i}}}k)\rVert_{1,t}\cdot \lvert t\rvert.
\end{split}
\end{equation*}

Now it remains to estimate the individual parts of equation (\ref{eqn: Telescoping sum trick}). First, by Proposition~\ref{prop: Dolgopyat Estimate} and our choice of $\rho_0$ we have  
\begin{equation}
\label{eqn: Telescope Bound 1}
     \lVert\mathcal{L}_{\psi-scr}^{n-m}\rVert_{1,t}\leq C\rho^{p\lfloor\log\left\lvert t\right\rvert\rfloor}e^{lP(\psi-\sigma cr)}< C\rho_0^{p\lfloor\log\left\lvert t\right\rvert\rfloor}e^{lP(\psi-\sigma cr)},
\end{equation}
where $n-m=p\lfloor\log\left\lvert t\right\rvert\rfloor+l$, with $0\leq l\leq\log\left\lvert t\right\rvert-1$. Moreover, since $|f'|\geq\gamma^{-1}$ with $\gamma\in (0,1)$, we have 
\begin{equation}
\label{eqn: Telescope Bound 2}
    d(x_{\underline{i}},x_{\underline{j}(\underline{i})})\leq\gamma^m
    < \rho_0^m
\end{equation}
since these points lie in the same inverse branch of $f^m$. It remains to estimate $\lVert \mathcal{L}_{\psi-scr}^m(\chi_{I_{\underline{i}}}k)\rVert_{1,t}$ when $|\underline{i}|=m$. Notice first that since the transfer operator sums over all inverse branches of $f^m$, the only summand which will make the term $\chi_{\underline{i}}$ non-zero is the inverse branch corresponding to the multi-index $\underline{i}$. Therefore we have the formula
\begin{equation*}
    \mathcal{L}_{\psi-scr}^m(\chi_{I_{\underline{i}}}k)=e^{(g^{(m)}-scr^{(n)})\circ f^{-m}_{\underline{i}}}(k\circ f^{-m}_{\underline{i}}).
\end{equation*}
We have the obvious bound
\begin{equation*}
    \left\lvert e^{(g^{(m)}-scr^{(m)})\circ f^{-m}_{\underline{i}}}(k\circ f^{-m}_{\underline{i}})\right\rvert _{C^0}\leq \left\lvert e^{g^{(m)}-scr^{(m)}}\right\rvert _{C^0} \cdot \left\lVert k\right\rVert _{C^0}.
\end{equation*}
Using the product and chain rules, together with the fact that $\left\lvert f^{-m}_{\underline{i}}\right\rvert_{C^1}\leq\gamma^m$, we have
\begin{equation*}
    \begin{split}
      &\left\lvert e^{(g^{(m)}-scr^{(m)})\circ f^{-m}_{\underline{i}}}(k\circ f^{-m}_{\underline{i}})\right\rvert _{C^1} \\
    &\ \leq \left\lvert e^{g^{(m)}-scr^{(m)}}\right\rvert_{C^0} \cdot \left\lvert k\circ f^{-m}_{\underline{i}}\right\rvert_{C^1}+\left\lvert e^{(g^{(m)}-scr^{(m)})\circ f^{-m}_{\underline{i}}}\right\rvert_{C^1}\cdot \left\lVert k\right\rVert_{C^0} \\
    &\ \leq \left\lvert e^{g^{(m)}-scr^{(m)}}\right\rvert_{C^0}\cdot \left\lVert k\right\rVert_{C^1} \cdot \gamma^m
    +\left\lvert e^{g^{(m)}-scr^{(m)}}\right\rvert_{C^0}\cdot \left\lvert (g^{(m)}-scr^{(m))}\circ f^{-m}_{\underline{i}}\right\rvert_{C^1} \cdot \left\lVert k\right\rVert_{C^0}.
    \end{split}
\end{equation*}
By definition,
\begin{equation*}
   (g^{(m)}- scr^{(m)})\circ f^{-m}_{\underline{i}}=\sum_{i=0}^{m-1}(\psi-r)\circ(f^i\circ f^{-m}_{\underline{i}}),
\end{equation*}
and since $f^i\circ f^{-m}_{\underline{i}}$  is an inverse branch of $f^{m-i}$, we have $|f^i\circ f^{-m}_{\underline{i}}|_{C^1}\leq\gamma^{m-i}$. Therefore
\begin{equation*}
\begin{split}
    \left\lvert (g^{(m)}- scr^{(m)})\circ f^{-m}_{\underline{i}}\right\rvert_{C^1}&\leq \sum_{i=0}^{m-1}(\|g\|_{C^1}+\lvert s\rvert c\|r\|_{C^1})\cdot \gamma^{m-i}\\
    &\leq C'(\|g\|_{C^1}+ \left\lvert t\right\rvert \cdot c\| r\|_{C^1})\\
    &\leq C'(\|g\|_{C^1}+  c\| r\|_{C^1})\left\lvert t\right\rvert.
\end{split}
\end{equation*}

Absorbing the norms $\|r\|_{C^1}$ and $\|g\|_{C^1}$ into $C'$, we get
\begin{equation*}
     \left\lvert e^{g^{(m)}-scr^{(m)}\circ f^{-m}_{\underline{i}}}(k\circ f^{-m}_{\underline{i}})\right\rvert_{C^1} \leq
    C' \left\lvert e^{g^{(m)}-scr^{(m)}}\right\rvert_{C^0} \cdot \left(\left\lvert t\right\rvert \cdot \left\lVert k\right\rVert_{C^0}+\left\lVert k\right\rVert_{C^1}\right).
\end{equation*}
Thus
\begin{equation*}
    \begin{split}
        \frac{\left\lvert \mathcal{L}_{\psi-scr}^m(\chi_{I_{\underline{i}}}k)\right\rvert_{C^1}}{\left\lvert t\right\rvert} &\leq C' \left\lvert e^{g^{(m)}-scr^{(m)}}\right\rvert_{C^0}\left(\left\lVert k\right\rVert_{C^0}+\frac{\left\lVert k\right\rVert_{C^1}}{\left\lvert t\right\rvert}\right)\\ 
        &\leq
    C' \left\lvert e^{g^{(m)}-scr^{(m)}}\right\rvert_{C^0} \cdot \left\lVert k\right\rVert_{1,t}.
    \end{split}
\end{equation*}
Putting this all together, we have the bound
\begin{equation}
\label{eqn: Telescope Bound 3}
    \left\lVert \mathcal{L}_{\psi-scr}^m(\chi_{I_{\underline{i}}}k)\right\rVert_{1,t}\leq
     C' \left\lvert e^{g^{(m)}-scr^{(m)}}\right\rvert_{C^0} \cdot \left\lVert k\right\rVert_{1,t}.
\end{equation}
Using the bounds (\ref{eqn: Telescope Bound 1}), (\ref{eqn: Telescope Bound 2}), and (\ref{eqn: Telescope Bound 3}) in (\ref{eqn: Telescoping sum trick}), we get the following bound 
\begin{equation*}
    \sum_{m=2}^n C \rho_0^{p\lfloor\log\left\lvert t\right\rvert\rfloor} e^{lP(\psi-\sigma cr)}\rho_0^m \cdot \left\lvert t\right\rvert \cdot \sum_{\left\lvert \underline{i}\right\rvert=m}\left\lvert e^{g^{(m)}-scr^{(m)}}\right\rvert_{C^0}\cdot  \left\lVert k\right\rVert_{1,t}
\end{equation*}
There exists a constant $C''>0$ such that
\begin{equation}
\label{eqn: Bound on sum over multiindices}
    \sum_{\left\lvert \underline{i}\right\rvert =m}\left\lvert e^{g^{(m)}-scr^{(m)}}\right\rvert _{C^0} \leq C'' e^{mP(\psi-\sigma cr)}\leq C''e^{nP(\psi-\sigma cr)},
\end{equation}
since for $\sigma<1$ we have $P(\psi-\sigma cr)>0$.

By the definitions of $p$ and $l$, $p\lfloor\log\left\lvert t\right\rvert\rfloor+m=n-l$, and $l<\log\left\lvert t\right\rvert$, we have

\begin{equation*}
    \sum_{m=2}^n C \rho_0^{n-l}\left\lvert t\right\rvert ^{P(\psi-\sigma cr)} \left\lvert t\right\rvert\sum_{\left\lvert \underline{i}\right\rvert=m}\left\lvert e^{g^{(m)}-scr^{(m)}}\right\rvert_{C^0} \left\lVert k\right\rVert_{1,t}\leq 
    Cn \left\lVert k\right\rVert_{1,t} \cdot \left(\rho_0 e^{P(\psi-\sigma cr)}\right)^n \left\lvert t\right\rvert^{1+P(\psi-\sigma cr)-\log\rho_0}.
\end{equation*}
Since the mapping $\sigma\mapsto P(\psi-\sigma cr)$ is continuous and decreasing, and since $P(\psi-~cr)=~0$, if we take $\sigma_1$ sufficiently close to $1$ with $\sigma_0<\sigma_1<1$, then $P(\psi-\sigma cr)\leq 1$ for every $\sigma_1\leq \sigma<1$. By taking $\sigma_1$ even closer to $1$ if necessary, we can assume that $\rho_0 e^{P(\psi-\sigma_1 cr)}<~1$. Letting $\rho_1:=~\rho_0 e^{P(\psi-\sigma_1 cr)}$, we get the desired conclusion of Lemma~\ref{lem: Key Bound for Eta} by taking $\varepsilon=~P(\psi-~\sigma_1cr)-~\log\rho_0$, which can be made arbitrarily close to $0$ by taking both $\sigma_1$ and $\rho_0$ sufficiently close to $1$.

\end{proof}
We can now use Lemma~\ref{lem: Key Bound for Eta} to bound $\eta_{\Psi,K}(s)$. We have
\begin{equation*}
    \begin{split}
        |Z_n(\psi-scr,k)| &\leq\left|Z_n(\psi-scr,k)-\sum_{i=1}^N\mathcal{L}_{\psi-scr}^n(\chi_i k)(x_i)\right|+\left|\sum_{i=1}^N\mathcal{L}_{\psi-scr}^n(\chi_i k)(x_i)\right| \\
        &\leq C_1||k||_{1,t}\left\lvert t\right\rvert^{1+\varepsilon}n\rho_1^n+\sum_{i=1}^N||\mathcal{L}_{\psi-scr}^n||_{1,t}||k||_{1,t} \\
        &\leq C_1||k||_{1,t}\left\lvert t\right\rvert^{1+\varepsilon}n\rho_1^n+C_2N||k||_{1,t}\rho_0^{p\lfloor\log\left\lvert t\right\rvert\rfloor}e^{lP(\psi-\sigma cr)},
    \end{split}
\end{equation*}
where $n=p\lfloor\log\left\lvert t\right\rvert\rfloor+l$. We now estimate
\begin{equation*}
    \begin{split}
        |\eta_{\Psi,K}(s)| &\leq \sum_{n=1}^\infty\frac{1}{n}|Z_n(\psi-scr,k)|\\
        &\leq C_1||k||_{1,t}\left\lvert t\right\rvert^{1+\varepsilon}\sum_{n=1}^\infty\rho_1^n+
     C_2N||k||_{1,t}\sum_{p=0}^\infty\rho_0^{p\lfloor\log\left\lvert t\right\rvert\rfloor} \sum_{l=0}^{\lfloor\log\left\lvert t\right\rvert\rfloor-1}e^{lP(\psi-\sigma cr)}\\
     &\leq C_1'||k||_{1,t}\left\lvert t\right\rvert^{1+\varepsilon}+C_2N||k||_{1,t}\left(\frac{1}{1-\left\lvert t\right\rvert^{-\log\rho_0}}\right)\left(\frac{e^{P(\psi-\sigma cr)\lfloor\log\left\lvert t\right\rvert\rfloor}-1}{e^{P(\psi-\sigma cr)}-1}\right).
    \end{split}
\end{equation*}

For sufficiently large $\left\lvert t\right\rvert$, the term $\left(\frac{1}{1-\left\lvert t\right\rvert^{-\log\rho_0}}\right)$ is uniformly bounded, and we have
\begin{equation}
\label{eqn: First Bound on Eta Growth}
     C_1'||k||_{1,t}\left\lvert t\right\rvert^{1+\varepsilon}+C_2'N||k||_{1,t}\left\lvert t\right\rvert^{P(\psi-\sigma cr)}\leq C_3||k||_{1,t}\left\lvert t\right\rvert^{1+\varepsilon}.
\end{equation}
Next, let $\delta>0$ be arbitrary. Then on the half line $\Re(s)=1+\delta$ we have the bound
\begin{equation*}
\begin{split}
    |\eta_{\Psi,K}(s)|&\leq \sum_{n=1}^\infty\frac{1}{n}\sum_{f^n(x)=x}|k^{(n)}(x)|e^{\psi^{(n)}(x)-(1+\delta)r^{(n)}(x)}\\
    &\leq||k||_{C^0}\sum_{n=1}^\infty e^{-n\delta(\inf r)}\sum_{f^n(x)=x}e^{\psi^{(n)}(x)-cr^{(n)}(x)}.
\end{split}
\end{equation*}
Since $\sum_{f^n(x)=x}e^{\psi^{(n)}-cr^{(n)}(x)}\leq Ce^{nP(\psi-cr)}=C$, we can bound the equation above by
\begin{equation}
\label{eqn: Bound For Eta to the right of entropy}
    C||k||_{C^0}\sum_{n=1}^\infty e^{-n\delta(\inf r)}\leq C_4||k||_{C^0}\leq C_4||k||_{1,t}.
\end{equation}
To conclude the proof of Theorem \ref{thm: Uniform extension of eta}, we will use the Phragm\'en-Lindel\"of theorem to interpolate between the bounds (\ref{eqn: First Bound on Eta Growth}) and (\ref{eqn: Bound For Eta to the right of entropy}). Namely, if we fix a $\sigma_2$ such that 
$\sigma_1<~\sigma_2<~1<1+\delta$, and let $\alpha(\sigma)$ be the linear function taking values $\alpha(\sigma_1)=1+\varepsilon$ and $\alpha(h+\delta)=0$, then the Phragm\'en-Lindel\"of theorem~\cite[\S2, p.150]{Iwaniec_Kowalski_2004} tells us that
\begin{equation}
    \label{eqn: Phragmen-Lindelof theorem Interpolation}
    |\eta_K(\sigma_2+it)|=O(||k||_{1,t}\left\lvert t\right\rvert^{\alpha(\sigma_2)}).
\end{equation}

By taking $\delta$ sufficiently close to $0$, we see that there exists such a $\sigma_2<1$ with $\alpha(\sigma_2)<1$. Taking this $\sigma:=\sigma_2$ and $\alpha:=\alpha(\sigma_2)$ finishes the proof of Theorem~\ref{thm: Uniform extension of eta}.

\begin{proof}[Proof of Theorem \ref{thm: Effective Equidistribution}]
We will now deduce effective equidistribution from Theorem \ref{thm: Uniform extension of eta}. Here, we change our normalization of $K$. Instead of assuming that $\int Kd\mu_{\Psi}=0$, we instead add a constant to $K$ so that $K\geq0$, noting that we can do this without changing our region or estimates in Theorem \ref{thm: Uniform extension of eta}.
Let
\begin{equation*}
    \Phi_{\Psi,K,0}(T):=\sum_{e^{hr^{(n)}(x)}\leq T}e^{\psi^{(n)}(x)}k^{(n)}(x),
\end{equation*}
and
\begin{equation*}
    \Phi_{\Psi,K,1}(T):=\int_1^T\Phi_{\Psi,K,0}(s)ds=\sum_{e^{hr^{(n)}(x)}\leq T}e^{\psi^{(n)}(x)}k^{(n)}(x)\left(T-e^{hr^{(n)}(x)}\right).
\end{equation*}
Notice that since $K\geq0$, $\Phi_{\Psi,K,0}(T)$ and $\Phi_{\Psi,K,1}(T)$ are increasing functions of $T$.
We will use the identity
\begin{equation*}
    \label{eqn: Inverse Mellin Transform}
    \frac{1}{2\pi i}\int_{d-i\infty}^{d+i\infty}\frac{y^{s+1}}{s(s+1)}ds=
     \begin{cases} 
      0 & 0<y\leq 1 \\
      y-1 & y>1, \\

\end{cases}
\end{equation*}
valid for any $d>1$, to relate $\Phi_{\Psi,K,1}$ and $\eta_{\Psi,K}$ as follows:
\begin{equation*}
    \Phi_{\Psi,K,1}(T)=\int_{d-i\infty}^{d+i\infty}\eta_{\Psi,K}(s)\frac{T^{s+1}}{s(s+1)}ds.
\end{equation*}
By Theorem \ref{thm: Uniform extension of eta}, $\eta_{\Psi,K}(s)$  is analytic in the half-plane $\Re(s)\geq 1-2\delta$ for some $\delta>0$, except for a simple pole at $s=1$ with residue
\begin{equation*}
    \text{Res}_{s=1}(\eta_{\Psi,K}(s))=\frac{1}{c}\int Kd\mu_\Psi.
\end{equation*}
Thus, by shifting the contour of integration, we have
\begin{equation*}
    \int_{d-i\infty}^{d+i\infty}\eta_{\Psi,K}(s)\frac{T^{s+1}}{s(s+1)}ds=
    \left(\frac{1}{c}\int Kd\mu_\Psi\right)\cdot\frac{T^2}{2}+\int_{(1-2\delta)-i\infty}^{(1-2\delta)+i\infty}\eta_{\Psi,K}(s)\frac{T^{s+1}}{s(s+1)}ds,
\end{equation*}
and along this domain of integration we have $\left\lvert \eta_{\Psi,K}(s)\right\rvert=O\left(\left\lVert K\right\rVert_{C^1} \left\lvert t\right\rvert^{\alpha}\right)$, for a fixed $0<\alpha<1$. Therefore,
\begin{equation*}
    \begin{split}
       \left|\int_{(1-2\delta)-i\infty}^{(1-2\delta)+i\infty}\eta_{\Psi,K}(s)\frac{T^{s+1}}{s(s+1)}ds\right|&\leq\int_{-\infty}^\infty|\eta_{\Psi,K}(\sigma+it)|\frac{T^{2-2\delta}}{|1-2\delta+it||2-2\delta+it|}dt\\
       &\leq C||K||_{C^1}T^{2-2\delta}\int_1^\infty\frac{t^\alpha}{t^2}dt\\
       &\leq C_5||K||_{C^1}T^{2-2\delta}.
    \end{split}
\end{equation*}
Hence we have the estimate
\begin{equation}\label{eqn: Preliminary Estimate on psi1}
    \Phi_{\Psi,K,1}(T)=\left(\frac{1}{c}\int Kd\mu_\Psi\right)\cdot\frac{T^2}{2}+O(\|K\|_{C^1}T^{2-2\delta})
\end{equation}
We would now like to pass from equation (\ref{eqn: Preliminary Estimate on psi1}) to a similar estimate on $\Phi_{\Psi,K,0}$. \\
Let $\Delta(T):=T^{1-\delta}$. Since $\Phi_{\Psi,K,0}$ is monotonically increasing, we have
\begin{equation*}
    \frac{\Phi_{\Psi,K,1}(T+\Delta(T))-\Phi_{\Psi,K,1}(T)}{\Delta(T)}=
    \frac{1}{\Delta(T)}\int_T^{T+\Delta(T)}\Phi_{\Psi,K,0}(s)ds\geq\Phi_{\Psi,K,0}(T).
\end{equation*}
By (\ref{eqn: Preliminary Estimate on psi1}), we have
\begin{equation*}
\begin{split}
   & \frac{\Phi_{\Psi,K,1}(T+\Delta(T))-\Phi_{\Psi,K,1}(T)}{\Delta(T)}\\  
   &\ = \frac{1}{2T^{1-\delta}}\left(\frac{1}{c}\int Kd\mu_\Psi\right)\left((T+T^{1-\delta})^2-T^2+
    O\left(||K||_{C^1}\left(T+T^{1-\delta})^{2-2\delta}\right)\right)-O\left(||K||_{C^1}T^{2-2\delta}\right)\right)\\
    &\ =\frac{1}{2T^{1-\delta}}\left(\frac{1}{c}\int Kd\mu_\Psi\right)\left(2T^{2-\delta}+T^{2-2\delta} \right)+O\left(||K||_{C^1}T^{1-\delta}\right)\\
    &\ =\left(\frac{1}{c}\int Kd\mu_\Psi\right)\cdot T+O\left(||K||_{C^1}T^{1-\delta}\right),
\end{split}
\end{equation*}
which gives us the upper bound $\left\lvert \Phi_{\Psi,K,0}(T)\right\rvert \leq \left(\frac{1}{c}\int Kd\mu_\Psi\right)\cdot T+ O\left(||K||_{C^1}T^{1-\delta}\right)$. Similarly we obtain the lower bound $\left\lvert \Phi_{\Psi,K,0}(T)\right\rvert \geq \left(\frac{1}{c}\int Kd\mu_\Psi\right)\cdot T+ O\left(||K||_{C^1}T^{1-\delta}\right)$ from the inequality
\begin{equation*}
     \frac{\Phi_{\Psi,K,1}(T)-\Phi_{\Psi,K,1}(T-\Delta(T))}{\Delta(T)}=
    \frac{1}{\Delta(T)}\int_{T-\Delta(T)}^{T}\Phi_{\Psi,K,0}(s)ds\leq\Phi_{\Psi,K,0}(T).
\end{equation*}
We therefore conclude that 
\begin{equation}
    \label{eqn: Estimate on psi0}
    \left\lvert \Phi_{\Psi,K,0}(T)\right\rvert=\left(\frac{1}{c}\int Kd\mu_\Psi\right)\cdot T+ O\left(||K||_{C^1}T^{1-\delta}\right).
\end{equation}
We are now ready to prove Theorem \ref{thm: Effective Equidistribution}. Let $K:M\rightarrow\mathbb{R}$ be $C^1$ such that $\int Kd\mu=0$. Then, by the definition of the measure $\mu_{T}$
\begin{equation*}
    \int Kd\mu_{\Psi,T}=\frac{1}{Z_T(g_t,\Psi)}\sum_{\tau\in\per((0,T])}e^{\int_\tau \Psi dt}\cdot\int_\tau Kdt.
\end{equation*}
If the orbit $\tau$ is coded by a point $x\in I$ with $f^n(x)=x$ and $r^{(n)}(x)=\per(\tau)$, and we have
\begin{equation*}
    \int_\tau Kdt=k^{(n)}(x),\hspace{1mm}\int_\tau \Psi dt=\psi^{(n)}(x),
\end{equation*}
where $k,\psi\in C^1(I)$ are as in Lemma \ref{lem: Lifting C1 functions to coding}. Technically, correspondence between periodic orbits of $g_t$ and periodic orbits of $f$ is not typically one-to-one, but it is close, in the sense that the amount of over counting we get in our coding is a subexponential amount. For simplicity, let us assume that the coding is perfect, so 
\begin{equation*}
    \begin{split}
        \int Kd\mu_{\Psi,T} &=\frac{1}{Z_T(g_t,\Psi)}\sum_{\tau\in\per((0,T])}e^{\int_\tau \Psi dt}\cdot\int_\tau Kdt\\ 
    &=\frac{1}{Z_T(g_t,\Psi)}\sum_{r^{(n)}(x)\leq T}e^{\psi^{(n)}(x)}k^{(n)}(x)\\
    &=\frac{1}{Z_T(g_t,\Psi)}\Phi_{\Psi,K,0}\left(e^{c T}\right),\\
    \end{split}
\end{equation*}
where we recall that $c:=P(\Psi)>0$.
Observe that 
\begin{equation*}
    Z_T(g_t,\Psi)=\sum_{\tau\in\per((0,T])}e^{\int_\tau \Psi dt}\cdot\ell(\tau)=\Phi_{\Psi,1,0}(e^{c T})=\frac{1}{c}\cdot e^{cT}+O(e^{(1-\delta)c T}).
\end{equation*}
Hence we have
\begin{equation}
    \int Kd\mu_\Psi^T=\frac{\Phi_{\Psi,K,0}\left(e^{c T}\right)}{\Phi_{\Psi,1,0}\left(e^{c T}\right)}=\frac{\frac{1}{c}\left(\int Kd\mu_\Psi\right)\cdot e^{cT}+O(\|K\|_{C^1}e^{(1-\delta)cT})}{\frac{1}{c}\cdot e^{cT}+O(e^{(1-\delta)c T})}=\int Kd\mu_\Psi+O(\|K\|_{C^1}e^{-\delta\cdot P(\Psi)\cdot T}),
\end{equation}
as desired. To account for the overcounting in the coding we instead write
\begin{equation*}
     \int Kd\mu_{\Psi,T}=\frac{1}{\Phi_{\Psi,1,0}(e^{c T})}\Phi_{\Psi,K,0}\left(e^{c\cdot T}\right)+\frac{1}{\Phi_{\Psi,1,0}(e^{c T})}\Phi_{\Psi,K,0}\left(e^{c\cdot T}\right)-\frac{1}{Z_T(g_t,\Psi)}\sum_{\tau\in\per((0,T])}e^{\int_\tau \Psi dt}\cdot\int_\tau Kdt.
\end{equation*}
The proof is finished if we can show that
\begin{equation*}
    \left|\frac{1}{\Phi_{\Psi,1,0}(e^{c T})}\Phi_{\Psi,K,0}\left(e^{c\cdot T}\right)-\frac{1}{Z_T(g_t,\Psi)}\sum_{\tau\in\per((0,T])}e^{\int_\tau \Psi dt}\cdot\int_\tau Kdt
     \right|=O\left(||K||_{C^1}e^{(c-\delta_0) T}\right),
\end{equation*}
for some $\delta_0>0$. 

We have 
\begin{equation}
\label{eqn: Overcounting Estimates}
\begin{split}
    & \left|\frac{1}{\Phi_{\Psi,1,0}(e^{c T})}\Phi_{\Psi,K,0}\left(e^{c\cdot T}\right)-\frac{1}{Z_T(g_t,\Psi)}\sum_{\tau\in\per((0,T])}e^{\int_\tau \Psi dt}\cdot\int_\tau Kdt  \right|\\
    &\ \leq\frac{1}{\Phi_{\Psi,1,0}(e^{c T})}\left|\Phi_{\Psi,K,0}\left(e^{c\cdot T}\right)-\sum_{\tau\in\per((0,T])}e^{\int_\tau \Psi dt}\cdot\int_\tau Kdt \right|\\
    &\ + \left|\sum_{\tau\in\per((0,T])}e^{\int_\tau \Psi dt}\cdot\int_\tau Kdt\right|\left|\frac{1}{\Phi_{\Psi,1,0}(e^{c T})}-\frac{1}{Z_T(g_t,\Psi)} \right| \\
    &\ \leq \frac{1}{\Phi_{\Psi,1,0}(e^{c T})}\left|\Phi_{\Psi,K,0}\left(e^{c\cdot T}\right)-\sum_{\tau\in\per((0,T])}e^{\int_\tau \Psi dt}\cdot\int_\tau Kdt \right|
    +\|K\|_{C^0}\frac{\left|\Phi_{\Psi,K,0}\left(e^{c\cdot T}\right)-Z_T(g_t,\Psi) \right|}{\Phi_{\Psi,K,0}\left(e^{c\cdot T}\right)}
\end{split}
\end{equation}
The two terms of \ref{eqn: Overcounting Estimates} are handled in exactly the same manner, so we will focus on estimating

\begin{equation*}
\begin{split}
    \left|\Phi_{\Psi,K,0}\left(e^{c\cdot T}\right)-\sum_{\tau\in\per((0,T])}e^{\int_\tau \Psi dt}\cdot\int_\tau Kdt \right|
    &=
    \left|\sum_{r^{(n)}(x)\leq T}e^{\int_\tau \Psi dt}\cdot\int_{\mathcal{O}(x)} Kdt- \sum_{\tau\in\per((0,T])}e^{\int_\tau \Psi dt}\cdot\int_\tau Kdt \right|\\
    &=\left|\sum_{\tau\in A_T}e^{\int_\tau \Psi dt}\cdot\int_\tau Kdt \right|\\
    &\leq ||K||_{C^0}T\sum_{\tau\in A_T}e^{\int_\tau \Psi dt},
\end{split}
\end{equation*}
where $A_T$ is the set of orbits of length at most $T$ in $(\Sigma_A^r,\sigma_t^r)$ which are ``redundant" for counting the orbits of $g_t$.

The following result of Bowen~\cite{Bowen_1973}, generalizing an earlier result of Manning~\cite{Manning} tells us how the weighted sums over $A_T$ grow:
\begin{prop}
\label{prop: Over counting Estimate}
    Let $c:=P(\Psi)>0$ be the pressure of the potential $\Psi:M\to \mathbb{R}$. Then there exists $\delta_0>0$ such that 
    \begin{equation*}
        \sum_{\tau\in A_T}e^{\int_\tau \Psi dt}=O\left(e^{(c-\delta_0)T}\right).
    \end{equation*}
\end{prop}
The over-counted orbits comprising $A_T$ are precisely the periodic orbits that lie on the boundary of the Markov partition elements. 
The standard Bowen-Manning approach is then to capture this over-counting on the boundary by constructing finitely many auxiliary subshifts $f_i:I\to I$ determined explicitly by the given Markov partition. Moreover, the pressure of $\Psi$ restricted to each these subshift satisfies $P(\Psi,f_i)<P(\Psi)$ (see Parry-Pollicott \cite{Parry_Pollicott_1990} Lemma 9.3 for a precise statement in the weighted case). In particular, one derives a precise factorization of the flow zeta function as the subshift Zeta function times a rational expression of the Zeta functions of the auxiliary subshfits. Proposition \ref{prop: Over counting Estimate} then follows from standard arguments, using Perron's formula (c.f. \cite[Theorem~5.6]{Iwaniec_Kowalski_2004}).

Thus
\begin{equation*}
\begin{split}
     \left|\Phi_{\Psi,K,0}\left(e^{c\cdot T}\right)-\sum_{\tau\in\per((0,T])}e^{\int_\tau \Psi dt}\int_\tau Kdt
     \right| &= O\left(||K||_{C^0}T\sum_{\tau\in A_T}e^{\int_\tau \Psi dt}\right)\\
     &=O\left(||K||_{C^0}Te^{(c-\delta_0)T}\right),
\end{split}
\end{equation*}
for any $0<\delta_1<\delta_0$, which finishes the proof.
\end{proof}

In Theorem \ref{thm: Effective Equidistribution} we considered the rate of equidistribution of all periodic orbits of length at most $T$. However, in some cases it is of interest to take some $\varepsilon>0$ and consider those periodic orbits with $T-\varepsilon< \ell(\tau)\leq T$. Define
\begin{equation}
\label{eqn: Weighted discrete measures in a window}
    \mu_\Psi^{(t-\varepsilon,T]}:=\frac{1}{Z_{(T-\varepsilon,T]}(g_t,\Psi)}\sum_{\tau\in\per((T-\varepsilon,T])}e^{\int_\tau \Psi }\delta_\tau
\end{equation}
Observing that 
\begin{equation*}
    \sum_{T-\varepsilon<r^{(n)}(x)\leq T}e^{\psi^{(n)}(x)}k^{(n)}(x)=\Phi_{\Psi,K,0}(e^{cT})-\Phi_{\Psi,K,0}(e^{c(T-\varepsilon)}),
\end{equation*}
and likewise 
\begin{equation*}
    Z_{(T-\varepsilon,T]}(g_t,\Psi)=\Phi_{1,K,0}(e^{cT})-\Phi_{1,K,0}(e^{c(T-\varepsilon)}),
\end{equation*}
We have the following immediate corollary:
\begin{cor}
    Let $g_t:M\rightarrow M$ be a contact Anosov flow on a three-dimensional manifold $M$, and let $\Psi\in C^1(M)$ be a potential with equilibrium state $\mu_\Psi$ and $P(\Psi)>0$, and let $ \mu_\Psi^{(t-\varepsilon,T]}$ be as in (\ref{eqn: Weighted discrete measures in a window}). Then there exist constants $C,\delta>0$ such that for any $C^1$ function $K:M\rightarrow\mathbb{R}$ we have
    \begin{equation}
    \label{eqn: Effective Equidistribution exp}
        \left|\int K d\mu_\Psi-\int K d \mu_\Psi^{(t-\varepsilon,T]} \right|\leq
        C||K||_{C^1}e^{-\delta\cdot P(\Psi)\cdot T}.
    \end{equation}
\end{cor}

\section{General Anosov Flows}\label{sec:effective-general-Anosov}
We now address the problem of extending Theorem \ref{thm: Effective Equidistribution} from  from contact Anosov flows on three-dimensional manifolds to more general Anosov flows. In fact, we only used that fact that we were working with a contact Anosov flows on a three-dimensional manifold in order to apply Proposition~\ref{prop: Dolgopyat Estimate}, which in turn allowed to show that our function $\eta_{\Psi,K}$ has an analytic extension to a half-line $\Re(s)\geq 1-\delta$ for some $\delta>0$, except for the simple pole at $s=1$ (see Theorem \ref{thm: Uniform extension of eta}). Our proof extends to any Anosov flows for which we have the estimate (\ref{eqn: Dolgopyat Estimate}), and hence can prove the corresponding version of Theorem \ref{thm: Uniform extension of eta}. It is an open problem to determine if the zeta function (even just the standard one-variable zeta function) for a weak-mixing Anosov flow has an analytic extension to a strictly larger half-plane.

In this section, following the approach of Pollicott and Sharp~\cite{Pollicott_Sharp_2001}, we will prove Theorem \ref{thm: Effective Equidistribution - general Anosov}, showing that for a transitive weak-mixing Anosov flow, we have a polynomial rate of equidistribution.

The proof proceeds along similar lines to that of Theorem \ref{thm: Effective Equidistribution}; however, since we do not in general have a ``$C^1$-coding'' as we do in the contact case, we will rely on the standard Bowen-Ratner coding given by Proposition \ref{prop: Bowen Ratner Coding}, and all of our functions will lift to $\alpha_0$-H\"older functions on the coding space, as explained in \S\ref{subsec:lifting-functions}.

For simplicity, we shall assume that $\Psi=0$ so that $\mu_\Psi=\mu$ is the measure of maximal entropy, with $P(0)=h>0$. Suppose that $K:M\to\mathbb{R}$ is a Lipschitz function. Consider the function $\eta_{K}$ defined over the shift $\sigma:\Sigma_A\to\Sigma_A$ given by
\begin{equation*}
    \eta_{K}(s)=\sum_{n=1}^\infty\frac{1}{n}\sum_{\sigma^n(x)=x}k^{(n)}(x)e^{-shr^{(n)}(x)}.
\end{equation*}
By Proposition \ref{prop: Initial Domain of eta}, $\eta_K$ is analytic in the half-plane $\Re(s)\geq 1$, except for a simple pole at $s=1$ with residue $\frac{1}{h}\int Kd\mu$. The next step is to extend the domain of analyticity to a larger region which is independent of $K$. Unfortunately, we will not be able to extend to a strictly larger half plane because we do not have Dolgopyat's estimate~\ref{eqn: Dolgopyat Estimate} at our disposal. Instead, we use in the following weaker estimate that appears in \cite[Proposition~$2$]{Pollicott_Sharp_2001} for the operator norm of $\mathcal{L}_{-shr}$ acting on the space of H\"older functions:
\begin{prop}[Pollicott-Sharp estimate]
    \label{prop: Weaker Dolgopyat estimate}
    There exist constants $t_0\geq 1$, $\tau>0$, $C_1>0$ and $C>0$ such that if $s=\sigma+it$ with $\left\lvert t\right\rvert\geq t_0$ and every $m\in\mathbb{N}$ we have
    \begin{equation*}
        \label{eqn: Weaker Dolgopyat Estimate}
        \left\lVert \mathcal{L}_{-shr}^{2m\lfloor C\log\left\lvert t\right\rvert\rfloor}\right\rVert \leq C_1 \left\lvert t\right\rvert e^{2m\lfloor C\log\left\lvert t\right\rvert \rfloor P(-\sigma r)}\left(1-\frac{1}{\left\lvert t\right\rvert^{\tau}}\right)^{m-1}.
    \end{equation*}
\end{prop}

If $n=2m\lfloor C\log\left\lvert t\right\rvert\rfloor+\ell$, with $0\leq\ell<2\lfloor C\log\left\lvert t\right\rvert\rfloor-1$, then we immediately get the bound
\begin{equation}
\label{eqn: General Weak Dolgopyat Estimate}
    \left\lVert \mathcal{L}^n_{-shr}\right\rVert \leq \left\lVert \mathcal{L}_{-shr}^{2m\lfloor C\log\left\lvert t\right\rvert\rfloor}\right\rVert \cdot \left\lVert \mathcal{L}_{-shr}^{\ell}\right\rVert \leq C_2\left\lvert t\right\rvert^{2} e^{nP(-\sigma r)}\left(1-\frac{1}{\left\lvert t\right\rvert^{\tau}}\right)^{m-1},
\end{equation}
for some $C_2>0$.
Letting 
\begin{equation*}
    Z_n(-shr,k):=\sum_{\sigma^n(x)=x}k^{(n)}(x)e^{-shr^{(n)}(x)},
\end{equation*}
we have as before
\begin{equation*}
    \eta_K(s)=\sum_{n=1}^\infty\frac{1}{n}Z_n(-shr,k).
\end{equation*}
We then have the following weaker version of Lemma~\ref{lem: Key Bound for Eta}:
\begin{lemma}[Ruelle-type estimate]
    \label{lem: Key Bound for Eta Weaker Version}
    There exist constants $C_3>0$ and $t_1\geq t_0$ such that for any points $x_i\in I_i$, all $n\in\mathbb{N}$, and all $\left\lvert t\right\rvert\geq t_1$, we have
    For  $\rho_0$ such that $\max\{\rho,\gamma\}<\rho_0<1$, there exist $C_{1},\epsilon$ and $0<\rho_1<1$ such that for  and $\left\lvert t\right\rvert\geq 1$ we have

    \begin{equation}
        \left\lvert Z_n(-shr,k)-\sum_{i=1}^N\mathcal{L}_{-shr}^n(\chi_{I_{i}}k)(x_i) \right\rvert \leq C_3 \left\lVert k\right\rVert_{\alpha_0}\left\lvert t\right\rvert^{3} e^{nP(-\sigma hr)}\left(1-\frac{1}{\left\lvert t\right\rvert^{\tau}}\right)^{\lfloor\frac{n}{2\lfloor C\log\left\lvert t\right\rvert\rfloor}\rfloor}.
        \end{equation}
\end{lemma}
We omit the proof as it is the same as that of Lemma 3 in \cite{Pollicott_Sharp_2001}, with additional step of tracking dependency on the test function $k$ just as we did in the proof of Lemma \ref{lem: Key Bound for Eta}. Using Lemma~\ref{lem: Key Bound for Eta Weaker Version}, we get the following analytic extension \cite[Corollary~$4.1$]{Pollicott_Sharp_2001}:
\begin{theorem}
    \label{thm: Weaker Analytic Extension of Eta}
    There exist constants $t_2\geq t_1$ and $\rho>0$ such that $\eta_K(s)$ has an analytic extension to the set
    \begin{equation*}
        \mathcal{R}(\rho):=\left\{s=\sigma+it\in\mathbb{C} \hspace{1mm}\bigg|\hspace{1mm} \sigma>\frac{1}{2h^{\rho+1}\left\lvert t\right\rvert^{\rho}},\left\lvert t\right\rvert\geq t_2 \right\}
    \end{equation*}
    Moreover, on $\mathcal{R}(\rho)$ we have the estimate
    $|\eta_K(s)|\leq C_4||K||_{Lip}\left\lvert t\right\rvert^{3(\rho+1)}$ for some $C_4>0$.
\end{theorem}
The only difference between Corollary 4.1 of \cite{Pollicott_Sharp_2001} and Theorem \ref{thm: Weaker Analytic Extension of Eta} is the tracked dependency of our estimates on $K$.

As before, we deduce effective equidistribution by studying the asymptotics of the function
\begin{equation*}
  \psi_{K,0}(T):= \sum_{e^{hr^{(n)}(x)}\leq T}k^{(n)}(x).
\end{equation*}
We will do this by studying the asymptotics of the auxiliary functions
\begin{equation*}
    \psi_{K,\ell}(T):=\int_1^T\psi_{K,\ell-1}(u)du=\sum_{e^{hr^{(n)}(x)}\leq T}k^{(n)}(x)\left(T-e^{hr^{(n)}(x)} \right)^\ell, \ \ell\in\mathbb{N}.
\end{equation*}
We will use the identity
\begin{equation*}
    \frac{1}{2\pi i}\int_{d-i\infty}^{d+i\infty}\frac{y^{s}}{s(s+1)\cdots(s+\ell)}ds=
     \begin{cases} 
      0 & 0<y\leq 1 \\
      \frac{1}{k!}\left(1-\frac{1}{y}\right)^\ell & y>1, \\

\end{cases}
\end{equation*}
valid for any $d>1$, to relate $\psi_{K,\ell}$ to $\eta_K$:
 \begin{equation}
     \psi_{K,\ell}(T)=\frac{1}{2\pi i}\int_{d-i\infty}^{d+i\infty}\eta_K(s)\frac{T^{s+\ell}}{s(s+1)\cdots(s+\ell)}ds.
 \end{equation}

We will estimate $\psi_{K,\ell}$ by changing the contour of integration so that we may use the bounds in Theorem \ref{thm: Weaker Analytic Extension of Eta}. The first step is to truncate the integral at some finite imaginary part. Fix $(3\rho+2)^{-1}<\varepsilon<\rho^{-1}$, where $\rho$ is as in Theorem \ref{thm: Weaker Analytic Extension of Eta}, and set $R=R(T)=(\log(T))^\varepsilon$. Then we have
\begin{equation*}
    \left|\psi_{K,\ell}(T)-\frac{1}{2\pi i}\int_{d-iR(T)}^{d+iR(T)}\eta_K(s)\frac{T^{s+\ell}}{s(s+1)\cdots(s+\ell)}ds\right|\leq 
    |\eta_{K}(d)|\frac{T^{d+\ell}}{\pi \ell R(T)^\ell}.    
\end{equation*}
We choose $d=1+\frac{1}{\log(T)}$, and first observe that
$T^{d+\ell}=T^{\ell+1}T^{\frac{1}{\log (T)}}=eT^{\ell+1}$. We claim that
$|\eta_K(d)|=O(||K||_{C^0}\log(T))$. Indeed we have
\begin{equation*}
    \begin{split}
        |\eta_K(d)| &\leq \sum_{n=1}^\infty\frac{1}{n}\sum_{\sigma^n(x)=x}|k^{(n)}(x)|e^{-hr^{(n)}(x)}e^{\frac{-h\cdot(\inf r)\cdot n}{\log(T)}}\\
        &\leq
    \left\lVert K\right\rVert _{C^0}\sum_{n=1}^\infty e^{\frac{-h\cdot(\inf r)\cdot n}{\log(T)}}\sum_{\sigma^n(x)=x}e^{-hr^{(n)}(x)} \\
    &\leq C\left\lVert K\right\rVert_{C^0}\sum_{n=1}^\infty e^{\frac{-h\cdot(\inf r)\cdot n}{\log(T)}}\\ 
    &\leq \frac{C\left\lVert K\right\rVert_{C^0}}{1- e^{\frac{-h\cdot(\inf r)}{\log(T)}}},
    \end{split}
\end{equation*}
and it follows from direct computation that
\begin{equation*}
    \frac{1}{1- e^{\frac{-h\cdot(\inf r)}{\log(T)}}}=O(\log(T)).
\end{equation*}
Thus
\begin{equation}\label{eqn: Contour Cutoff Estimate}
    \begin{split}
        \left|\psi_{K,\ell}(T)-\frac{1}{2\pi i}\int_{d-iR(T)}^{d+iR(T)}\eta_K(s)\frac{T^{s+\ell}}{s(s+1)\cdots(s+\ell)}ds\right|
    &= O\left(||K||_{C^0}\frac{T^{\ell+1}\log(T)}{(\log(T))^{\ell\varepsilon}}\right)\\
    &= O\left(||K||_{C^0}\frac{T^{\ell+1}}{(\log(T))^{\ell\varepsilon-1}}\right). 
    \end{split}
\end{equation}

By the residue theorem we have 
\begin{equation*}
    \int_{d-iR(T)}^{d+iR(T)}\eta_K(s)\frac{T^{s+\ell}}{s(s+1)\cdots(s+\ell)}ds=
    \left(\frac{1}{h}\int Kd\mu \right)\cdot\frac{T^{\ell+1}}{(\ell+1)!}+
    \int_\Gamma\eta_K(s)\frac{T^{s+\ell}}{s(s+1)\cdots(s+\ell)}ds,
\end{equation*}
where $\Gamma$ is the contour consisting of the directed line segments 
$[d+iR(T),C(R)+iR(T)]$, $[C(R)+iR(T),C(R)-iR(T)]$, and $[C(R)-iR(T),d-iR(T)]$, where we choose $C(R):=~1-~\frac{1}{h^{\rho+1}R(T)^\rho}$. Notice that by the choice of $C(R)$, the contour $\Gamma$ lies entirely within the region $\mathcal{R}(\rho)$ for sufficiently large $T$.

We first consider the integral over $[d+iR(T),C(R)+iR(T)]$:
\begin{equation}
\label{eqn: Horizontal Contour Estimate}
\begin{split}
    &\left|\int_{d+iR(T)}^{C(R)+iR(T)}\eta_K(s)\frac{T^{s+\ell}}{s(s+1)\cdots(s+\ell)}ds\right|\\
    &\ \leq\int_{C(R)}^d |\eta_K(\sigma+iR(T))|\frac{|T^{\sigma+\ell+iR(T)}|}{|\sigma+iR(T)||\sigma+1+iR(T)|\cdots|\sigma+\ell+iR(T)|}d\sigma \\
    &\ \leq  C_4||K||_{Lip}R(T)^{3(\rho+1)}\int_{C(R)}^d\frac{T^{\sigma+\ell}}{R(T)^{\ell+1}}d\sigma\\
    &\ \leq C_4||K||_{Lip}R(T)^{3(\rho+1)}(d-C(R))\frac{T^{\ell+d}}{R(T)^{\ell+1}}\\
    &\ = O\left(||K||_{Lip}\frac{T^{\ell+1}}{(\log T)^{(\ell+1-3(\rho+1))\varepsilon}}\right),
\end{split}
\end{equation}
where we used the bounds $|\eta_K(\sigma+iR(T))|\leq C_4||K||_{Lip}|R(T)|^{3(\rho+1)}$ and $d-C(R)=O(1)$. We will choose $\ell$ such that 
$\ell+1-3(\rho+1)>0$. The integral over the segment $[C(R)-iR(T),d-iR(T)]$ is bounded in exactly the same manner.

It remains to bound the integral over the segment $[C(R)+iR(T),C(R)-iR(T)]$. We break this up further into the segments $[C(R)+iR(T),C(R)+it_2]$, $[C(R)+it_2,C(R)-it_2]$, and $[C(R)-it_2,C(R)-iR(T)]$. On the first and last of these segments, we can use the bound on $\eta_K$ in $\mathcal{R}(\rho)$, whereas since the size of the middle segment $[C(R)+it_2,C(R)-it_2]$ does not grow with $T$, we immediately can see that the integral over this segment is $O(||K||_{Lip})$. The remaining two segments are handled identically. We have
\begin{align*}
    &\left|\int_{C(R)+iR(T)}^{C(R)+it_2}\eta_K(s)\frac{T^{s+\ell}}{s(s+1)\cdots(s+\ell)}ds\right|\\
    &=\left|\int_{t_2}^{R(T)}\eta_K(C(R)+it)\frac{T^{C(R)+\ell+it}}{(C(R)+it)(C(R)+1+it)\cdots(C(R)+\ell+it)}dt \right|\\
    &\leq C_4||K||_{Lip}T^{C(R)+\ell}\int_{t_2}^{R(T)}t^{3(\rho+1)-(\ell+1)}dt\\
    &\leq C_5||K||_{Lip}T^{C(R)+\ell}R(T)^{3(\rho+1)-\ell}.
\end{align*}
Observe that
\begin{equation*}
    T^{C(R)+\ell}=T^{\ell+1-\frac{1}{h^{\rho+1}R(T)^\rho}}=T^{\ell+1}e^{\frac{-\log T}{h^{\rho+1}(\log T)^{\rho\varepsilon}}}=T^{\ell+1}e^{\frac{-(\log T)^{1-\rho\varepsilon}}{h^{\rho+1}}},
\end{equation*}
 and $e^{\frac{-(\log T)^{1-\rho\varepsilon}}{h^{\rho+1}}}=O((\log T)^{-\gamma})$ for every $\gamma>0$. Thus
 \begin{equation}
 \label{eqn: Vertical Contour Estimate}
     \left|\int_{C(R)+iR(T)}^{C(R)+it_2}\eta_K(s)\frac{T^{s+\ell}}{s(s+1)\cdots(s+\ell)}ds\right|=O\left(||K||_{Lip}\frac{T^{\ell+1}}{(\log T)^\gamma}\right),
 \end{equation}
 for every $\gamma>0$.

Now comparing (\ref{eqn: Contour Cutoff Estimate}), (\ref{eqn: Horizontal Contour Estimate}), and (\ref{eqn: Vertical Contour Estimate}), and noting that by the choice of $\varepsilon$ we have 
$(3(\rho+1)-\ell-1)\varepsilon< \ell\varepsilon-1$, we conclude that
\begin{equation}
\label{eqn: Auxiliary Function Estimate}
\begin{split}
    \psi_{K,\ell}(T)&=\left(\frac{1}{h}\int Kd\mu \right)\cdot\frac{T^{\ell+1}}{(\ell+1)!}+O\left(||K||_{Lip}\frac{T^{\ell+1}}{(\log T)^{(\ell+1-3(\rho+1))\varepsilon}}\right)\\
    &=\left(\frac{1}{h}\int Kd\mu \right)\cdot\frac{T^{\ell+1}}{(\ell+1)!}+O\left(||K||_{Lip}\frac{T^{\ell+1}}{(\log T)^{\beta}}\right),
\end{split}
\end{equation}
where $\beta:=(\ell+1-3(\rho+1))\varepsilon>0$.

We now want to pass from (\ref{eqn: Auxiliary Function Estimate}) to an asymptotic bound on $\psi_{K,0}(T)$. The process is the same as in the preceding section, except we take $\Delta(T):=T(\log T)^{-\beta/2}$, and estimate 
\begin{align*}
    \psi_{K,\ell-1}(T)&\leq\frac{1}{\Delta(T)}\int_T^{T+\Delta}\psi_{K,\ell-1}(u)du=\frac{\psi_{K,\ell}(T+\Delta(T))-\psi_{K,\ell}(T)}{\Delta(T)}\\
    &=\left(\frac{1}{h}\int Kd\mu \right)\cdot\frac{T^{\ell}}{\ell!}+O\left(||K||_{Lip}\frac{T^\ell}{(\log T)^{\beta/2}}\right),
\end{align*}
and the lower bound follows by the same arguments. Repeating this argument yields the bound
\begin{equation}
    \label{eqn: Asympotics for psi in Anosov}
    \psi_{K,0}(T)=\left(\frac{1}{h}\int Kd\mu \right)\cdot T+O\left(||K||_{Lip}\frac{T}{(\log T)^\delta}\right),
\end{equation}
where $\delta=\beta/2^\ell>0$.

We can now deduce Theorem \ref{thm: Effective Equidistribution - general Anosov}:
\begin{proof}[Proof of Theorem \ref{thm: Effective Equidistribution - general Anosov}]
    Assume for now that the coding is perfect and $\int Kd\mu=0$. Then
    \begin{equation*}
        \int Kd\mu_{T}=\frac{1}{Z_T(f_t)}\psi_{K,0}(e^{hT})=O\left(||K||_{Lip}\frac{e^{hT}}{(\log e^{hT})^\delta}\cdot e^{-hT}\right)=O\left(||K||_{Lip}T^{-\delta}\right),
    \end{equation*}
as desired. The process of dealing with over counting in the counting is handled in exactly the same way as in the proof of Theorem \ref{thm: Effective Equidistribution}.

\end{proof}

\printbibliography
\end{document}